\documentclass[makeidx]{amsart}
\usepackage{epsfig}
\usepackage{amsmath}
\usepackage{amssymb}
\newtheorem{Proposition}{Proposition}
  \newtheorem{Remark}[Proposition]{Remark}
  \newtheorem{Corollary}[Proposition]{Corollary}
  \newtheorem{Lemma}[Proposition]{Lemma}
  
  \newtheorem{Theorem}[Proposition]{Theorem}
 
 \newtheorem{Definition}[Proposition]{Definition}
 \newtheorem{Note}[Proposition]{Note}

\newcommand {\z}{{\noindent}}
 
\def\CC{\mathbb{C}}
 \def\RR{\mathbb{R}}
 \def\NN{\mathbb{N}}
\def\ZZ{\mathbb{Z}}

\def\Re{\mathrm{Re\,}}
\def\Im{\mathrm{Im\,}}
\def\ds{\displaystyle}
 \def\({\left(} \def\){\right)} \makeindex
\author{O. Costin, M. Huang and F. Fauvet} \address{Mathematics Department The Ohio State University Columbus, OH 43210} \address{Department of Mathematics,
The University of Chicago, 5734 S. University Avenue
Chicago, Illinois 60637}\address{IRMA, Universit\'e de Strasbourg et CNRS, 67084 Strasbourg, France} \title{Global behavior of solutions of nonlinear ODEs: first order equations}
\begin{document}
\gdef\shorttitle{Global behavior of 1st order ODEs}
\maketitle
$ $ \vskip -0.2cm
\begin{abstract}
  We determine the behavior of the general solution, small or large,
  of nonlinear first order ODEs in a neighborhood of an irregular
  singular point chosen to be infinity.  We show that the solutions can be
  controlled in a ramified neighborhood of infinity using a finite set of
  asymptotic constants of motion; the asymptotic formulas can be
  calculated to any order by quadratures. These constants of motion
  enable us to obtain qualitative and accurate quantitative
  information on the solutions in a neighborhood of infinity, as well
  as to determine the position of their singularities. We discuss how
  the method extends to higher order equations. There are some
  conceptual similarities with a KAM approach, and we discuss this
  briefly.
\end{abstract}
\section{Introduction}
The point at infinity is most often an {\it irregular singular point}
for equations arising in applications.\footnote{A
  singular point of an equation is irregular if, for {\em small}
  solutions, the linearization is not of Frobenius type. By a small
  solution we mean one that tends to zero in some direction after
  simple changes of coordinates.} Within this class of equations, there
  are essentially two types for which a global description of
  solutions exists: linear systems and integrable ones. However, in a
  stricter sense, even for some linear problems global questions such
  as explicit values of connection coefficients are still open. The
  behavior of the general solutions of {\em linear} ODEs has been
  thoroughly analyzed starting in the late 19th century, see
  \cite{Fabry} and \cite{Wasow} and references therein.  After the
  pioneering work of \'Ecalle, Ramis, Sibuya and others the
  description of their solutions in $\CC$ is by now quite well
  understood \cite{Ecalle,Ecalle-book,Balser, Balser3, Braaksma,
    Ramis1, Ramis2, Duke}.

{\em Integrable} systems provide another important class of systems allowing
for global description of solutions. The ensemble of integrable
systems is a zero measure set in the parameter space of general
equations: a generic small perturbation of an integrable system
destroys integrability. Nonetheless, integrable equations occur
remarkably often in many areas of mathematics, such as orthogonal polynomials, the analysis of the
Riemann-zeta function,  random matrix theory,
self-similar solutions of integrable PDEs and combinatorics, cf.
\cite{Bleher},\cite{Deift3}--\cite{Deift1}, \cite{Ablowitz, Fokas3,
  Calogero, Conte, Deift1}, \cite{Fokas}--\cite{zak}. However, even in
integrable systems, achieving global control of solutions in a {\em
  practical way} is a challenging task, and it is one of the important
aims of the emerging Painlev\'e project \cite {Painleve22}.

In {\em nonintegrable} systems, particularly near irregular singularities,
our understanding is much more limited. Small solutions are given by
generalized Borel summable {\em transseries}; this was discovered by
\'Ecalle in the 1980s and proved rigorously in many contexts
subsequently. Transseries are essentially formal multiseries in
powers of $1/x^{k_i}$
$e^{-\lambda_j x}$,   and possibly $x^{-1}\log
x$; see again \cite{Ecalle,Ecalle-book,Balser, Balser3, Braaksma, Ramis1, Ramis2,  Duke} and \cite{OCBook}. Here $x$ is the independent variable and $\lambda_j$
are eigenvalues of the linearization with the property
$\Re (\lambda_j x)>0$.
In general, {\em only} small solutions are well understood. However,
for generic nonlinear systems of higher order, small solutions form
lower dimensional manifolds in the space of all solutions, see, e.g.,
\cite{Duke}. The present understanding of general nonlinear equations
is thus quite limited.

We introduce a new line of approach, combining ideas from generalized
Borel summability and KAM theory (see, e.g. \cite{Arnold}) for the
analysis near infinity, chosen to be an irregular singular point, of
solutions of relatively general differential equations with
meromorphic coefficients.  Applying the method does not require
knowledge of Borel summability, transseries or KAM theory.

\smallskip

\bigskip For small solutions, in \cite{Invent} it was shown that in a
region adjacent to the sector where the solution,  $y$, is small,
$y(x)$ is almost periodic. In this sense $y$ becomes an approximately
cyclic variable. In the $x$-complex plane, the singular points of $y$
are arranged in quasi-periodic arrays as well. The analysis in
\cite{Invent} covers an angularly small region beyond the sector
where $y$ is small.  Looking directly at the
asymptotics of $y$ beyond this region would require a multiscale
approach: $y$ has a periodic behavior--the fast scale, with $O(1/x)$
changes in the quasi-period.  Multiscale analysis is usually a quite
involved procedure (see, e.g., \cite{Bender}).

It is natural  to make a hodograph transformation
in which the dependent and independent variables are switched. As mentioned
above, in the ``nontrivial'' regions, the dependent variable is an almost cyclic one.
The setting becomes
somewhat similar to a KAM one: there is an
underlying completely integrable system, and one looks for persistence
of invariant tori. Adiabatic invariants are simply the conserved
quantities associated with these tori. Evidently there are many
differences between the ODE setting and the KAM one, for instance the fact
that the small parameter is ``internal'', $1/x$.

In this work we restrict the analysis to first order equations, mainly
to ensure a transparent and concrete analysis. In theory however, the
method generalizes to equations of any order, and we touch on these
issues at the end of the paper.

We look at equations which, after normalization, are of the form
$dy/dx=F(z,y)$, $z=1/x$, with $g$ bi-analytic at ${\bf 0}$ and $F_y({\bf
  0})=1$.

We show that in any sector on  Riemann surfaces towards infinity,  the
{\em general} solution is represented by transseries and/ or, in an
implicit form, by some constant of motion.
In fact, on large circles around $x=0$, the solution cycles
among transseries representations and ones in which constants
of motion describe it accurately. The regions where these behaviors occur
overlap slightly to allow for asymptotic matching (cf. Corollary \ref{C1}). The connection
between the large $x$ behavior and the initial condition is
relatively easy to obtain.

Let $\beta=F_{zy}({\bf 0})$. The constants of motion have asymptotic expansions of the form
\begin{equation}
  \label{eq:eqinit}
  C(x,y)\sim x-\beta\log x+F_0(y)+x^{-1}F_1(y)+\cdots+x^{-j}F_j(y)+\cdots, \ \ x\to\infty
.\end{equation}
 Clearly, under the assumptions above, the solution
$y$ can be obtained asymptotically from \eqref{eq:eqinit} and the
implicit function theorem.  The requirement that $C$ is to leading
order of the form $f_1(x)+f_2(y)$, determines $C$ up to trivial
transformations, see Theorem \ref{T1} and Note
\ref{N1}.

The functions $F_j$ are shown in the proof of Theorem \ref{T1} to
solve first order autonomous ODEs, and thus they can always be
calculated by quadratures.

To illustrate this, we use a nonintegrable  Abel equation,
\begin{equation}
  \label{eq:Abel}
  u'=u^3-t
.\end{equation}

We note that there is no consensus on
how nonintegrabilty should be defined; for (\ref{eq:Abel}), it is the case that the equation passes no criterion of integrability, including the
poly-Painlev\'e test, and that there are
no solutions known, explicit or coming from, say,  some associated Riemann-Hilbert reformulation.

The Abel equation has the normal form (see \S\ref{sable}, where further details about this example are given)
\begin{gather}
  \label{tr12}
y'+3y^3-\frac{1}{9}+\frac{1}{5x}y=0
.\end{gather}
Regions of smallness are those for which  $y$ approaches a
root of $3y^3-1/9$; in these regions, $y$ is given by a transseries \cite{Invent}. Otherwise, $y$ has an implicit representation of the form
\begin{multline}
\label{newton0}
y=\frac{1}{3}\exp\bigg(-C-x+\frac{1}{5}\log x+\left(\sqrt{3}-\frac{2\sqrt{3}}{5x}\right)\arctan \left(\frac{6y+1}{\sqrt{3}}\right)\\
-\log(3y-1)+\frac{1}{2}\log(9y^2+3y+1))+\frac{1}{x}\left(\frac{27y^2}{5(1-27y^3)}+\frac{1}{25}+O(1/x)\right)\bigg)+\frac{1}{3}
,\end{multline}
 obtained by inverting an appropriate constant of motion $C$ (see \eqref{newton});
for the values of $\beta, F_0, F_1$ see \S\ref{S4}.

While in a numerical approach to calculating solutions the precision
deteriorates as $x$ becomes large, the accuracy of \eqref{eq:eqinit}
instead, {\em increases}.  In examples, even when \eqref{eq:eqinit} is
truncated to two orders, \eqref{eq:eqinit} is strikingly close to the
actual solution even for relatively small values of the independent
variable, see e.g. Figure\,\,\ref{fig:abel4}.

The procedure allows for a convenient way
to link initial conditions to global asymptotic behavior, see
e.g. \eqref{asing}.

 \subsection{Solvability versus integrability}

 First order equations for which the associated second order
 autonomous system is Hamiltonian are in particular {\em
   integrable}. Indeed, by their definition, there is a globally
 defined smooth $H$ with the property that $\dot{x}\frac{\partial
   H}{\partial x}+\dot{y}\frac{\partial H}{\partial y}=0$, that is
 $H(x(t),y(t))=const$, providing a closed form implicit, global
 representation of $y$.  While the differential equation provides
 ``infinitesimal'' information, $H$ --effectively an integral--
 provides a global one.

 Conversely, clearly, if there exists an implicit solution of the
 equation or indeed a smooth enough conserved quantity, the equation
 comes from a Hamiltonian system.

 What we provide is a finite set of matching conserved quantities,
 analogous to an atlas of overlapping maps projecting the differential
 field onto the trivial one, $H'=0$. They give, in a sense, a {\em
   foliation of the phase space} allowing for global control of
 solutions.  With obvious adaptations, this picture extends to higher
 order systems. In integrable systems there is just one single-valued
 map and the field is globally rectifiable. In general, the conserved
 quantities may be branched and not globally defined.

\subsection{Normalization and definitions}\label{Sec11}  Many equations
of the form $y'=F(y,1/x)$ with $F$ analytic for small $y$ and
small $1/x$  can be brought to the normal form
$y'=P_0(y)+Q(y,1/x)$ by systematic changes of variables,  see {\em e.g.} \cite{Duke}, \cite{OCBook}.

The
assumptions are that $Q(y,z)$ is entire in $y$ and analytic
in $z$ for small $z$, and $O(y^2,yz^2,z)$ for small $y$ and $z$ and that
$P_0$ is a polynomial. We assume that the roots of $P_0$ are {\em simple}.
It will be seen from the analysis that a more general $P_0$
can be accommodated. We thus write the equation as
\begin{equation}
  \label{eq:eqy0}
  y'=\sum_{k=0}^{\infty}\frac{P_k(y)}{x^k}=Q_1(y,1/x)=P_0(y)+Q(y,1/x)
.\end{equation}

\label{D1}

\begin{Definition}\label{def1}
$\bullet$   A formal constant of motion of \eqref{eq:eqy0} for $x\to \infty$ in an unbounded domain $\mathcal{D}\subset\mathbb{C}^2$ or on a Riemann surface covering it, and in which to leading order in $1/x$ the variables $x$ and $y$ are separated additively is a formal series
 \begin{equation}
    \label{eq:eq0}
    \tilde{C}(y,x)=A(x) +F_0(y)+\frac{F_1(y)}{x}+\cdots+\frac{F_j(y)}{x^j}+\cdots
  \end{equation}
such that  we have
$$\frac{d}{dx}\tilde{C}(y(x),x)=O(x^{-\infty})$$
in the sense that, for any $j$, $F_j$ and $H_j$ defined by
\begin{equation}
  \label{eq:defC}
\frac{H_{j+1}(x,y)}{x^{j+1}}:=  A'(x) +D_x \left(F_0(y)+\frac{F_1(y)}{x}+\cdots+\frac{F_j(y)}{x^j}\right)
\end{equation}
are uniformly bounded in $\mathcal{D}$; here $D_x$ is the derivative
along the field, $$D_x F(x,y)=\nabla F\cdot
(1,Q_1)=F_x(x,y)+F_y(x,y)Q_1(y,1/x).$$ See also \eqref{formalexpans}
below.

$\bullet$ An actual  constant of motion associated to
$\tilde{C}$ in $\mathcal{D}\subset\CC^2$ is a function $C$ so that
$C(y,x)\sim\tilde{C}(y,x)$ as $x\to\infty$ and
$\frac{d}{dx}C(y(x),x)=0$ for all solutions in $\mathcal{D}$.
\end{Definition}
\begin{Note}\label{N1}
 It will be seen that there is rigidity in the form of the constant of motion: if the variables in $\tilde{C}$ are, to leading order,  separated additively as in \eqref{eq:eq0},
then, up to trivial transformations, we must have
\begin{equation}
  \label{eq:eqA}
  A(x)=-x+a\log x
\end{equation}
where $a$ is the same as the one in the transseries
expansion of the solution, see Proposition~\ref{trans}. \end{Note}

\bigskip
\subsection{Finding the terms in the expansion of $\tilde{C}$}
\z Using (\ref{eq:eqA})  and truncating \eqref{eq:eq0} at an arbitrary $n>2$, let
\begin{equation}
  \label{eq:def2c}
  {C}_n(y,x)=:-x+a \log x+F_0(y)+\sum_{k=1}^{n}\frac{F_k(y)}{x^k}
.\end{equation}
We can check that $D_x C_n$ satisfies
\begin{multline}\label{formalexpans}
D_x C_n=-1+P_0F_0'+\frac{a+P_1F_0'+P_0F_1'}{x}\\
+\sum_{k=2}^{n}\frac{(1-k)F_{k-1}+\sum_{j=0}^{k}P_{k-j}F'_{j}}{x^k}
+\frac{-nF_n+\sum_{j=0}^{n}\sum_{k=0}^{\infty}P_{n+k+1-j}F'_{j}x^{-k}}{x^{n+1}}
\end{multline}
(cf. \eqref{eq:eqy0}) where the numerator of the last term is $H_n$ by definition. In order for $\tilde{C}$ to be a formal constant of motion, the coefficients of $x^{-j},j=0,1,2,\ldots$ must vanish, giving
\begin{align}\label{refF}
F_0'(y)&=\frac{1}{P_0(y)}\\
F_1'(y)&=-\frac{a+P_1(y)F_0'(y)}{P_0(y)} \label{refF1}\\
\label{dfk}
F_k'(y)&=\frac{(k-1)F_{k-1}(y)-\sum_{j=0}^{k-1}P_{k-j}(y)F'_{j}(y)}{P_0(y)}\;\quad(2\leq k \leq n)
.\end{align}
It follows in particular that $F'_0\ne 0$ and $F_0$ is bounded  in $\mathcal{D}$. In solving the differential system, the constants of integration are chosen
so that $F_k$ are indeed uniformly bounded in $y$, see \eqref{gk}.
\subsection{Solving for $y(x)$} The expression $C_n$ is an approximate constant
of motion
; we thus can find an approximate solution $y_n$ by fixing $C_n=K$. We then write
\begin{equation}
  \label{eq:eqyn}
G(y;K):=  F_0(y)-K-x+a\log x+\sum_{k=1}^n \frac{F_k(y)}{x^k}=0
\end{equation}
and we note that in the domain relevant to us ($\mathcal{S}_1$, see Theorem \ref{regcom} below) the analytic implicit function theorem applies since
\begin{equation}
  \label{eq:difK}
\frac{\partial G}{\partial y}=\frac{1}{P_0(y)} +\frac{1}{x}E_1(y,x)
\end{equation}
where $P_0$ is away from $0$ in our domain, and for some $E_1$ which is bounded in $\mathcal{D}$ by \eqref{eq:def2c} since $Y$ is bounded.  Writing $y=y_n$ in \eqref{eq:eqyn} and using the analytic implicit function theorem, treating $1/x$ as a small parameter, we get

\begin{equation}
  \label{eq:eqy}
  y_n=G_0(x;K)+\frac{G_1(x;K)}{x}+\cdots+\frac{G_n(x;K)}{x^n} + \frac{\tilde{H}_n(x;K)}{x^{n+1}}
\end{equation}
where  $\tilde{H}_n$ and the $G_j$'s are bounded. In the
same way it is checked that $y_n$ is solution of \eqref{eq:eqy0} up to
corrections $R_n(x;K)/x^{n+1}$, that is,
$y_n'-Q_1(y_n,1/x)=-R_n(x;K)/x^{n+1}$ where $R_n$ is bounded.

Let $p_1,\ldots,p_m$ be the distinct roots of $P_0$.

Let $\mathcal{R}_y$ be
the universal cover of
$Y=\mathbb{C}\backslash\{p_1,...,p_m\}$. Let $\pi:X\rightarrow Y$ be
the covering map.
\begin{Definition}\label{Def4}
$\bullet$ {\rm
An {\bf elementary $\bf y$-path} of type

$$\alpha=(\alpha_1,...,\alpha_m,\alpha_{m+1},\ldots,\alpha_{mk})\in\mathbb{Z}^{mk}, k\in\NN$$ is a piecewise
smooth curve $\gamma$ in $\mathcal{R}_y$ whose image under $\pi$ turns $\alpha_{1}$ times around
$p_1$, then $\alpha_{2}$ times around $p_2$, and so on, $\alpha_{m}$
times around $p_m$, then again $\alpha_{m+1}$ times around $p_1$ , etc. Note that $\alpha$ is in fact an element of the fundamental group.

$\bullet$ A {\bf $\bf y$-path of type $\boldsymbol\alpha$ } is a smooth curve
$\gamma$ obtained as an arbitrary forward concatenation of elementary $
y$-paths of type $\alpha$. More precisely, a $ y$-path
of type $\alpha$ is a map $\gamma:[0,\infty)\to \mathcal{R}_y$ so that, for
any $N\in \mathbb{Z}^+$, $\gamma|_{[N,N+1]}$ is an elementary $
y$-path of type $\alpha$.  We will naturally denote by
$\gamma|_{[0,a]}$ subarcs of $\gamma$. We see that $y$-paths are
compositions of {\em closed loops} in the {\em complex $y$ domain.}

$\bullet$  $\mathcal{S}_r$ is  {\bf a regular domain of type $\boldsymbol\alpha$} or
an $R$-domain of type $\alpha$, if it is an unbounded open subset of
$\mathcal{R}_y$ that contains only images of $
y$-paths of type $\alpha$. Thus the image of any unbounded  $y$-path of type $\alpha'\neq \alpha$ is not a subset of $\mathcal{S}_r$.

{\bf Note.} In our results we only need $ y$-paths with the
additional property that $x(y)\to\infty$ along the path.

 To take a
trivial illustration, in the equation $y'=y$ an example of a $ y$-path
along which $x\to \infty$ is $t\mapsto \exp(it), t\ge 0$.

}\end{Definition}
\section{Main results}
\subsection{Existence of formal constants of motion} Under the assumptions at
the beginning of \S\ref{D1} we have
\begin{Theorem}
\label{regcom}
Let $\mathcal{S}_y$ be an $R$-domain of type $\alpha$, and
$$\mathcal{S}_1=\{y\in \mathcal{S}_y: |\pi(y)|< M_0 ~{\rm and}
~|\pi(y)-p_k|>\epsilon ~{\rm for~all~}k\}$$
where $M_0>0$ is an arbitrary constant.
Let $\mathcal{C}$ be the
union of $m$ circular paths surrounding $\alpha_k\in\ZZ$ times the root $p_k$, $k=1,\ldots,m$,  chosen so that
\begin{equation}
  \label{eq:eqnontr}
\int_{\mathcal{C}}\frac{1}{P_0(y)}dy\neq0
.\end{equation}
  Then, if $R_0$ is large enough, there exists a
formal constant of motion in
$$\mathcal{D}_1=\{(x,y):|x|>R_0,y\in\mathcal{S}_1\}$$ of the form
\eqref{eq:eq0}.
The terms $F_k$ in the expansion of $\tilde{C}$ in \eqref{eq:eq0} can be calculated
by quadratures.
\end{Theorem}
Actual constants of motion are obtained in Theorem \ref{T1}.

Consider now a set  $\mathcal{S}$ of curves  $\gamma$, $|\gamma(t)|\to\infty$ as $t\to\infty$, with the property that  for all  $t_1<t_2$ and all $n$ (which is in fact equivalent to for $n=0,1$ )
\begin{equation}
  \label{eq:restrgam}
  \left| \Re \displaystyle \int_{\gamma(t_1)}^{\gamma(t_2)} \frac{\partial}{\partial y}Q_1(y,\gamma(t))|_{y=y_n(\gamma(t))}\gamma'(t)dt\right|\leqslant b\log(|\gamma(t_2)/\gamma(t_1)|+1)
  ,\end{equation} where $b>0$ is a constant, and such that there is an $M$ so that  for all $n$ we have $|y_n|<M$ along $\gamma$. Here $M$ can be chosen large if $x$ is large. Note that $\mathcal{S}$ contains
the curves $\gamma(t)$ so that  $y_n(\gamma(t))$ is an
$\alpha$-path. Indeed, by \eqref{eq:eqy} ,  in this case, the integrand  in
\eqref{eq:restrgam} is of the form
$\frac{P_0'(y)}{P_0(y)}dy+O(1)\frac{d\gamma(t)}{\gamma(t)}$ and hence the integral equals  $2\pi i N +O(\log(|N|+1))$ for large $N$ where $N$ is the number of loops.

\begin{Theorem}\label{T1} Assume $\tilde{C}$ in \eqref{eq:eq0}
  is a formal constant of motion in a region $\mathcal{D}=\mathcal{S}\cap \mathcal{D}_1$. Then there exists an actual constant of motion $C=C(x,y)$
defined in the same region, so that $C\sim \tilde{C}$ as $x\to\infty$.

\end{Theorem}
\subsection{Regions where $P_0(u)$ is small}
Assume $x_0$ is large and $|P_0(y(x_0))|<\epsilon$ is sufficiently
small. This means that for some root $r_k$ of $P_0$ we have
$|y(x_0)-r_k|<\epsilon_1$ where $\epsilon_1$ is also small. Without
loss of generality we can assume that $r_k=0$ and $x\in \RR^+$ since
the change of variables $y_1=y-r_k$, $x=x_1 e^{i\phi}$ does not change
the form of the equation. Assume also that after normalization the
stability condition $\Re P_0'(0)<0$ holds. Again without loss of
generality, by taking $y_2=\alpha y_1$ we can arrange that $P_0'(0)=-1$.
 The new function $Q$ in \eqref{eq:eqy0} will have the form
$y^2Q_1(y,1/x)+x^{-2}Q_2(y,1/x)$ where $Q_1$ and $Q_2$ are analytic
for small $y$ and $1/x$. As a result, the normalized equation assumes
the form
\begin{equation}\label{nf}
  y'=-y+f_0(x)+\frac{ay}{x}+y^2Q_1(y,1/x)+x^{-2}Q_2(y,1/x)
.\end{equation}
We also arrange that $f_0=O(x^{-M})$  as $x\to\infty$, for suitably large $M$; this is possible
through a change of variables of the form $y_2=y_3+\sum_{k=1}^M c_k
x^{-k}$, where the $c_k$'s are the coefficients of the formal power
series solution for small $y$.
\begin{Proposition}
\label{trans}
[see \cite{Duke} Theorem 3] Any solution of \eqref{nf}  that is $o(1)$ as $x\to\infty$ along some
ray in the right half plane can be written as a
  Borel summed transseries, that is
\begin{equation}
  \label{eq:eqtrans}
  y(x)=\sum_{k=0}^{\infty}C^k x^{ka+1}e^{-kx} y_k
\end{equation}
where $y_k$ are generalized Borel
sums of their asymptotic series, and the decomposition is unique. There exist
bounds, uniform in $n$ and $x$, of the form $|y_n(x)|<A^k$, and thereby the sum converges uniformly
 in a region $\mathcal{R}$ that contains
any sector $\mathcal{S}_c:=\{x:|\arg\, x|<c<\pi/2\}$.
Note that Theorem 3 in  \cite{Duke} applies to general n-th order ODEs.
\end{Proposition}
\begin{Proposition}\label{ptranss}{
(i)  If, after the normalization above, $y(x_0)$ is small (estimates
can be obtained from the proof), then
$y$ is given by \eqref{eq:eqtrans}.

(ii)  $C(y(x),x)$, obtained by inversion of (\ref{eq:eqtrans}) for large $x$ in the right half plane and small $y$,  is a constant of motion defined for all
solutions for which $y(x_0)$ is small  (cf. (i)).}
\end{Proposition}
 \begin{proof}
(i) We write the differential equation in the equivalent integral form
\begin{multline}
  \label{eq:intfor}
  y=F_0(x)+y_0 e^{-(x-x_0)}(x/x_0)^a\\ + e^{-x}x^a \int_{x_0}^x e^s s^{-a} \left[ y^2(s)Q_1(y(s),1/s)+s^{-2}Q_2(y(s),1/s) \right]ds
,\end{multline}
where $F_0(x)=O(x^{-M})$ ($M$ can be chosen arbitrarily large in the normalization process, \cite{Duke}) and $F_0(x_0)=0$. It is straightforward to show that for
(\ref{eq:intfor}) is contractive in the norm $\|y\|=\sup_{x\in\mathcal{S}_c}|x^{M-1} y(x)|$
(see the beginning of this section) and thus it has a unique solution in this space. Hence, by uniqueness, the solution
of the ODE with  $y(x_0)=y_0$, has the property $y(x)\to 0$ as $x\to\infty$.  The rest of (i) now follows from
 \cite{Duke}.

 (ii)  We see from Proposition \ref{trans} that $y(x;C)$ is analytic
in a domain of the form  $\mathcal{S}_c\times \mathbb{D}_\rho$ (As usual, $\mathbb{D}_\rho$ denotes the disk of radius $\rho$.)
We look at the rhs of \eqref{eq:eqtrans} as a function $H(x,C)$.
It follows from \cite{Duke} that $y_1(x)=x^{-1}(1+o(1/x))$. By uniform convergence, we clearly have
\begin{equation}
  \label{eq:difh}
  \frac{\partial H}{\partial C}=\sum_{k=0}^{\infty}kC^{k-1} x^{ka+1}e^{-kx} y_k=e^{-x}x^a(1+o(1))\ne 0
.\end{equation}
The rest follows from the implicit function theorem.
\end{proof}
As a result of Theorem \ref{T1} and
Proposition \ref{ptranss} we have the following:
 \begin{Corollary}\label{C1}
 If $G_0$ in \eqref{eq:eqy} approaches a root of $P_0$ and $x$ is large enough, then
$y$ enters a transseries region, where the new  constant is given,
after normalization, by Proposition \ref{ptranss} (ii); thus the
constants of motion in different regions match.
 \end{Corollary}
\section{Proofs and further results}

\subsection{Proof of Theorem \ref{regcom}}
Let $(x_0,y_0)\in\mathcal{D}_1$. Recalling  (\ref{formalexpans}), we
see that \eqref{refF} has the solution
$$F_0(y)=\int_{y_0}^y\frac{1}{P_0(s)}ds+c_0$$
(we take $c_0=0$ since it can be absorbed into the constant of motion).  Eq. \eqref{refF1} gives
\begin{equation}\label{aaa}F_1(y)=f_1(y)+c_1:=-\int_{y_0}^y\frac{a+\frac{P_1(s)}{P_0(s)}}{P_0(s)}ds+c_1,\end{equation}
where to ensure boundedness of $F_1(y)$ as the number of loops $\to\infty$, we let
$$a=-\frac{\int_{\mathcal{C}}\frac{P_1(y)}{P_0(y)^2}dy}{\int_{\mathcal{C}}\frac{1}{P_0(y)}dy}$$
and $c_1$ is determined to ensure boundedness of $F_2$ (cf. \eqref{gk}). Inductively we have
\begin{equation}
\label{fk}
F_{k+1}(y)=\int_{y_0}^{y}\frac{k(f_k(s)+c_k)-\sum_{j=0}^{k}P_{{k+1}-j}(s)F'_{j}(s)}{P_0(s)}ds+c_{k+1}=:f_{k+1}(y)+c_{k+1}
\end{equation}
for $2\leq k+1\leq n$, and, to ensure boundedness of $F_{k+1}(y)$ as the number of loops $\to\infty$ we need to choose
\begin{equation}
\label{gk}
c_{k}=\dfrac{\int_{\mathcal{C}}\dfrac{-kf_{k}+\sum_{j=0}^{k}P_{k+1-j}(y)f'_{j}(y)}{P_0(y)}dy}{k\int_{\mathcal{C}}\dfrac{1}{P_0(y)}dy}
\end{equation}
for $1\leq k\leq n-1$.

It is clear by induction that every singularity of $F_k(y)$ is a root of $P_0$.
To complete the proof we need to show that the $F_k$'s are bounded in $\mathcal{D}_1$:
\begin{Lemma} Assume $y\in\mathcal{S}_1$.
For $\deg(P_0)\geq1$ and $1\leq k\leq n$ we have
$$|F_k'(y)|\lesssim k!\ \ \ \ \text{and}\ \ \ \ |F_k(y)|\lesssim k!(|y|+1)$$
where, as usual, $\lesssim $ means $\le $ up to an irrelevant multiplicative constant.
\end{Lemma}

\begin{proof}
We prove the lemma by induction on $k$. Note that in (\ref{aaa}) and (\ref{fk}) the integration paths can be decomposed into finitely many circular loops $\mathcal{C}$ and a ray, slightly deformed around possible singularities, which implies
$$|F_1(y)|\lesssim \log|y|+1\lesssim |y|+1$$
and
$$|F_k(y)|\lesssim \left|\int_{y_0}^{y}|F_k'(s)|ds\right|$$
where the integration path is a straight line (possibly  bent
as above).

We see from (\ref{dfk}) that
$$|F_k'(y)|\lesssim \frac{(k-1)|F_{k-1}(y)|}{|P_0(y)|}+\sum_{j=0}^{k-1}|F_j'(y)|\lesssim \frac{(k-1)|F_{k-1}(y)|}{|y|+1}+\sum_{j=0}^{k-1}|F_j'(y)|.$$
The conclusion then follows by induction. Note that the the last term of \eqref{formalexpans} satisfies
$$\left|-nF_n+\sum_{j=0}^{n}\sum_{k=0}^{\infty}P_{n+k+1-j}F'_{j}x^{-k}\right|\lesssim (n+1)!(|y|+1)|P_0(y)|$$

\end{proof}

\subsection{Proof of Theorem~\ref{T1}} Let $y(x;K)=y_n(x;K)+\delta(x;K)$, where $y_n$ is given in  \eqref{eq:eqy}. We seek  $\delta$ so that $y$ is an exact solution of \eqref{eq:eqy0} in $\mathcal{D}$.

 Let $\phi(y,\delta)$ be the polynomial satisfying $Q_1(y+\delta,x)-Q_1(y,x)=Q_{1,y}(y,x)\delta+\delta^2 \phi(y,\delta,x)$ where $Q_{1,y}(y,x):=\frac{\partial Q_1(y,x)}{\partial y}$. We obtain
\begin{equation}
  \label{eq:del}
  \delta'-\frac{b\delta}{x}-\frac{\partial Q_1(y,x)}{\partial y}\delta=\frac{R(x;K)}{x^{n+1}}-\frac{b\delta}{x}+\phi(y_n,\delta,x)\delta^2=:E(x;\delta(x);K)
,\end{equation}
 where  $R=:R_n$ is defined after \eqref{eq:eqy}; both $R$ and $\phi$ are, by assumption,
bounded. In integral form, \eqref{eq:del}  reads
\begin{equation}
  \label{eq:deli}
\delta(x)=\int_{\infty}^x \frac{x^b}{s^b}e^{\int_{s}^xQ_{1,y}(y_n(t),t)dt} E(s;\delta(s);K)ds
\end{equation}
where the integrals are taken along curves in $\mathcal{D}$.  Using \eqref{eq:restrgam} we see that  (\ref{eq:deli}) is contractive in the
norm $\|\delta\|=\ds \sup_{|x|\geqslant |x_1|;
  x\in\mathcal{D}}|x|^{n}|\delta(x)|$ in an arbitrarily large ball,  if $|x_1|$ is large enough and
$n>b2^{b+1}$.

Thus \eqref{eq:deli} has a unique solution and, of course, $\delta(x)$
is the limit of the Picard like iteration
\begin{multline}
  \label{eq:eqar}
\delta_0=\int_{\infty}^x \frac{x^b}{s^b} e^{\int_{s}^xQ_{1,y}(y_n(t),t)dt} \frac{R(s;K)}{s^{n+1}}ds\\
\delta_1=\int_{\infty}^x \frac{x^b}{s^b} e^{\int_{s}^xQ_{1,y}(y_n(t),t)dt} E(s;\delta_0(s);K) ds\\
etc.
\end{multline}
By \eqref{eq:deli} $\delta$ is a smooth function depending
on $(x, K)$ only, and  $\delta=O(x^{-n})$. Smoothness is shown as usual by bootstrapping the integral representation \eqref{eq:deli}.

Now we
have, by \eqref{eq:difK}, $\partial_K y_n(x;K)=P_0(y_n)(1+o(1))$. We
can easily check that $\partial_K\delta(x,K)=O(x^{-n})$. This is done
using essentially the same arguments employed to check contractivity
of the integral equation for $\delta$ in the equation in variations
for $\delta_K$, derived by differentiating \eqref{eq:del} with respect
to $K$. We use the implicit function theorem to solve for $K$, giving
$K=K(x,y)$, a smooth function of $(x,y)$. It has the following
properties: $K(x,y(x))$ is by construction constant along admissible
trajectories and by straightforward verification, i.e. comparing $K$
with $\tilde{C}$, we see that it is asymptotic to $\tilde{C}$ up to
$O(x^{-n})$. It is known that if a function differs from the $n$th
truncate of its series by $O(x^{-n})$ for large $n$, then in fact the
difference is $o(x^{-n})$ (cf.  \cite{OCBook} Proposition 1.13 (iii)).

\subsection{Position of singularities of the solution} It is convenient to introduce constants of motion specific to singular
regions; they provide a practical way to determine the position
of singularities, to all orders.
\begin{Definition}
We define a {\bf simple singular solution path} $\gamma(s):[0,1)\rightarrow \mathcal{R}_y$ to be a piecewise smooth curve whose projection $\pi(\gamma[0,1))\in \CC$ is unbounded but turns around every $p_k$ only finitely many times.

A {\bf simple singular solution domain} $\mathcal{S}_s$ is the homotopy class of any simple singular solution path, in the sense that any two unbounded paths in $\mathcal{S}_s$ can be continuously deformed into each other without passing through any $p_k$.
\end{Definition}

\begin{Proposition}\label{sincom}
Let $m_0=\deg(P_0)\geq 2$, $\mathcal{S}_s$ be a simple singular solution domain, and $\mathcal{D}_2=\{(x,y):|x|>R,y\in\mathcal{S}_s, ~{\rm  and} ~|y-p_k|>\epsilon ~{\rm for~all~}k\}$. Assume that $$\frac{|P_k(y)|}{|P_0(y)|}\lesssim |y|^{-q}$$ for large $y$, for some $q\geq 0$ and all $k\geq 1$. Note that this needs only be true in $\mathcal{S}_s$, which could be an angular region.

Then there exists in $\mathcal{D}_2$ a formal constant of motion of the form
\begin{equation}
    \label{eq:com1}
    \tilde{C}(y,x)=x+F_0(y)+\frac{F_1(y)}{x}+\cdots+\frac{F_j(y)}{x^j}+\cdots
  ,\end{equation}
 where $F_k(y)$ are single valued as $y\to\infty$.
Furthermore, any simple singular solution path passing through some arbitrary $(x_0,y_0)$ tends to a singularity, whose position $x_{sing}$ satisfies
\begin{equation}
\label{sing}
x_{sing}= C_n(y_0,x_0)+O\left(\frac{1}{x_0^{n+1}}\right)
\end{equation}
for all $n\in\mathbb{N}$, where $C_n$ is $\tilde{C}$ truncated to $x^{-n}$.

Moreover, if there are only finitely many nonzero $P_k$, then there exists in $\mathcal{D}_2$ a true constant of motion of the form (\ref{eq:com1}), i.e. the sum is convergent for large $|x|$.

\end{Proposition}

\begin{proof}
The proof is similar to that of Theorem \ref{regcom}.

In order for $\tilde{C}$ to be a formal constant of motion, we must have
\begin{align}
F_0'(y)&=-\frac{1}{P_0(y)}\\
\label{dfk1}
F_k'(y)&=\frac{(k-1)F_{k-1}(y)-\sum_{j=0}^{k-1}P_{k-j}(y)F'_{j}(y)}{P_0(y)}\;\quad(1\leq k\leq n).
\end{align}

We solve successively for the $F_k$  and obtain
$$F_0(y)=\int_{\infty}^y\frac{1}{P_0(s)}ds$$
where the integration path lies in $\mathcal{S}_s$. Clearly $F_0$ is bounded and single valued as $y\to\infty$.

Inductively we have
\begin{equation}
\label{fk1}
F_k(y)=\int_{\infty}^{y}\frac{(k-1)F_{k-1}(s)-\sum_{j=0}^{k-1}P_{k-j}(s)F'_{j}(s)}{P_0(s)}ds
\end{equation}
for $1\leq k\leq n$.

To prove the rest of the proposition, we need the following lemma:

\begin{Lemma} Assume that $y\in\mathcal{S}_s$.
For $1\leq k\leq n$ we have
$$|F_k'(y)|\lesssim \frac{k!}{|y|^{m_0+q}}$$
$$|F_k(y)|\lesssim \frac{k!}{|y|^{m_0+q-1}}$$
as $y\rightarrow\infty$.

Furthermore, if $P_k=0$ for $k>k_0>0$, then
$$|F_k'(y)|\lesssim \frac{c^k}{|y|^{m_0+q}}$$
$$|F_k(y)|\lesssim \frac{c^k}{|y|^{m_0+q-1}}.$$
\end{Lemma}

\begin{proof}
The estimates are obtained  by induction on $k$.
Note that (\ref{dfk1}) implies
$$|F_k'(y)|\lesssim \frac{(k-1)|F_{k-1}(y)|}{|y|^{m_0}}+|y|^{-q}\sum_{j=0}^{k-1}|F'_{j}(y)|$$
 provided that the assumptions of the lemma hold for $1\leq j\leq k-1$.

If $P_k=0$ for $k>k_0>0$, we again show the lemma by induction.

\z Assume that for $0<l\leq k$ we have
$$|F_l'(y)|\leq(c_0 l_0)^l \sum_{j=1}^{l+1}\binom {l}{j-1}|y|^{-1+j(1-m_0)-q}$$
(this is obviously true for $l=1$).

This implies
$$|F_k(y)|\leq(c_0 k_0)^k\sum_{j=1}^{k+1}\binom {k}{j-1}\frac{|y|^{j(1-m_0)-q}}{j(m_0-1)+q}$$

Thus it follows from (\ref{dfk1}) that

$$F_{k+1}'(y)=\frac{kF_{k}(y)-\sum_{j=\max\{k-k_0+1,0\}}^{k}P_{k+1-j}(y)F'_{j}(y)}{P_0(y)}$$
where, by the induction assumption, the first term satisfies  the estimate
\begin{multline}
\left|\frac{kF_{k}(y)}{P_0(y)}\right|\leq c_1 k \left|\frac{F_{k}(y)}{y^{m_0}}\right|\\
\leq c_0^k k_0^{k+1} c_1\sum_{j=2}^{k+2}\frac{(k+1)\binom {k}{j-2}}{(j-1)(m_0-1)}|y|^{-1+j(1-m_0)-q}\\
\leq c_0^k k_0^{k+1} c_1 \sum_{j=2}^{k+2}\binom {k+1}{j-1}|y|^{-1+j(1-m_0)-q}
\end{multline}
where $c_0>1+c_1$. Note that the last inequality follows from
$$(k+1)\binom {k}{j-2}=\binom {k+1}{j-1}(j-1).$$
The second term is easy to estimate, since it is clearly bounded by
$$k_0 (c_0 k_0)^k \sum_{j=1}^{k+1}\binom {k}{j-1}|y|^{-1+j(1-m_0)-q}.$$
Since $c_1$ is fixed, we can assume that $c_0>1+c_1$, and we have
$$|F_{k+1}'(y)|\leq(c_0 k_0)^{k+1} \sum_{j=1}^{k+2}\binom {k+1}{j-1}|y|^{-1+j(1-m_0)-q}.$$
This shows the second part of the lemma.
\end{proof}

Now since $$|D_x C_n|=\Big|\frac{-nF_n+\sum_{j=0}^{n}\sum_{k=0}^{\infty}P_{n+k+1-j}F'_{j}x^{-k}}{x^{n+1}}\Big|
\lesssim \frac{|P_0(y)|}{|x|^{n+1}|y|^{m_0+q}}$$ (cf. \ref{formalexpans}), the estimate for $x_{sing}$ follows immediately from integrating $D_xC_n$ from $x_0$ to $x_{sing}=C_n(\infty,x_{sing})$ along the simple singular solution path.

\end{proof}

\begin{Remark}
The condition $$\frac{|P_k(y)|}{|P_0(y)|}\lesssim |y|^{-q}$$ is not the most general one for which there exists a formal constant of motion in a simple singular domain. However, this condition is frequently satisfied by ODEs that occur in applications (see \S\ref{sable}). In such cases we can easily use (\ref{sing}) to find the position of the singularity (see e.g. (\ref{asing})).

\end{Remark}

\section {Example: the nonintegrable Abel equation \eqref{eq:Abel}}\label{sable}
To illustrate how to obtain information of the solution of a first order ODE using Theorem \ref{regcom} and Proposition \ref{sincom},
we take as an example the nonintegrable Abel equation
 \eqref{eq:Abel}. Normalization is achieved by
the transformation $x=-(9/5)A^2 t^{5/3}$, $A^3=1$, $u(t)=A^{3/5}(-135)^{1/5}x^{1/5}y(x)$ \cite{Invent}, yielding
\begin{equation}
\label{abel}
y'=-3y^3+\frac{1}{9}-\frac{y}{5x}.
\end{equation}

Obviously (\ref{abel}) satisfies the assumptions in Theorem \ref{regcom} and \ref{sincom},
 with $P_0(y)=-3y^3+\ds\frac{1}{9}$ and $P_1(y)=-\ds\frac{y}{5}$.

The three roots of $P_0$ are $\ds\frac{1}{3},~\ds\frac{(-1)^{2/3}}{3}$, and $\ds\frac{(-1)^{4/3}}{3}$. It is known \cite{Invent} that there exists a solution in the right half plane $\mathbb{H}$ that goes to the root $\ds\frac{1}{3}$ as $x\to\infty$. Similarly, there are solutions that go to the other two roots in other regions, which we will explore in \S\ref{phase}. In those cases, the behavior of the solution follows from Proposition \ref{trans} (see also \cite{Invent}). However, there are also solutions that do not go to any of the three roots. In these cases, the formal constant of motion will be a useful tool to describe quantitatively the behavior of the solution.

\subsection{Constants of motion in $R$-domains} \label{S4} (cf. Definition \ref{Def4}).
First we choose an elementary solution path along which the solution $y$ to (\ref{abel}) turns around the root $\ds\frac{1}{3}$ clockwise as shown in Fig \ref{fig:abel1} and \ref{fig:abel2}.

\begin{figure}[ht!]
\includegraphics[scale=0.5]{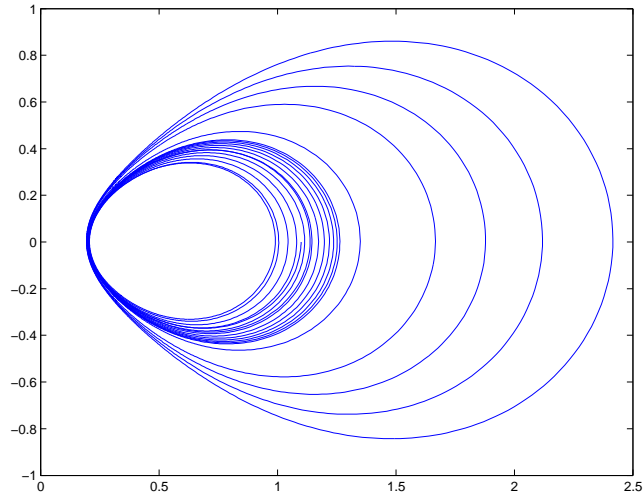}
\caption{Solution $y(x)$ with $y_0=1.1$ along the line segments from 1+5i to 1.5+50i to 1.6+120i}
\label{fig:abel1}
\end{figure}

\begin{figure}[ht!]
\includegraphics[scale=0.5]{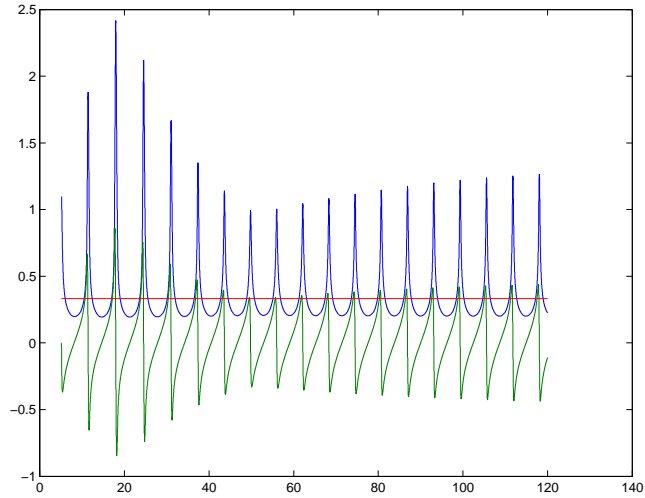}
\caption{Real and imaginary parts of $y(x)$. The upper curve is the real part, the lower curve is the imaginary part, and the straight line is the root $1/3$.}
\label{fig:abel2}
\end{figure}

For simplicity we calculate the first two terms of the expansion (\ref{eq:def2c}). We have
\begin{align}\label{f36}
F_0(y)&=\int\frac{1}{-3y^3+\frac{1}{9}}dy=\sqrt{3}\arctan \left(\frac{6y+1}{\sqrt{3}}\right)-\log(3y-1)+\frac{1}{2}\log(9y^2+3y+1)\\
a&=\dfrac{\int_{\mathcal{C}}\frac{y}{5(-3y^3+\frac{1}{9})^2}dy}{\int_{\mathcal{C}}\frac{1}{-3y^3+\frac{1}{9}}dy}=\frac{1}{5}\\
F_1(y)&=-\int\dfrac{\frac{1}{5}-\frac{y}{5(-3y^3+\frac{1}{9})}}{-3y^3+\frac{1}{9}}dy
=\frac{1}{10}\left(\frac{54y^2}{1-27y^3}-4\sqrt{3}\arctan \left(\frac{6y+1}{\sqrt{3}}\right)\right)+\frac{1}{25}
,\end{align}
where the constant $\ds\frac{1}{25}$ is found using (\ref{gk}).

We plot the first two orders of the formal constant of motion in Fig. \ref{fig:abel3}.

\begin{figure}[ht!]
\includegraphics[scale=0.6]{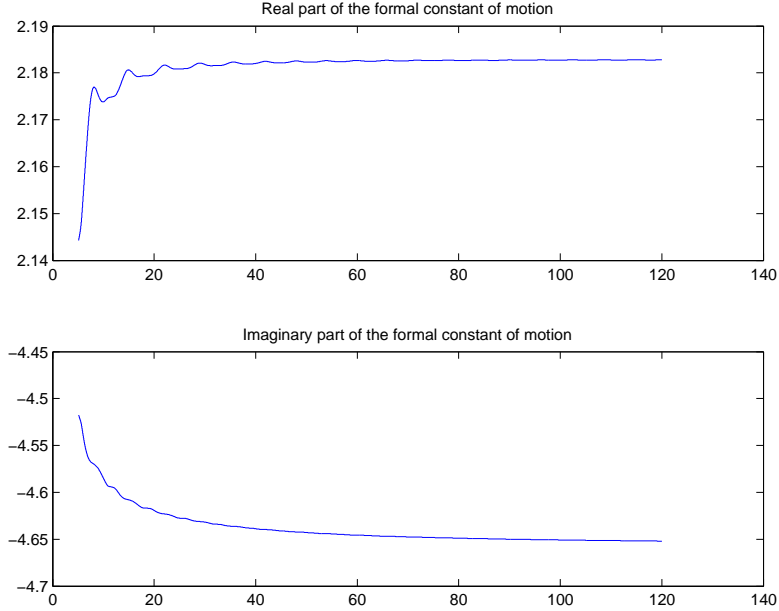}
\caption{Formal constant of motion with $F_0$ and $F_1$.}
\label{fig:abel3}
\end{figure}

Since this formal constant of motion is almost a constant along any path in the same $R$-domain, it can be used to find the solution
asymptotically, writing
$$C=-x+\frac{1}{5}\log x+F_0(y)+\frac{F_1(y)+O(1/x)}{x}.$$
 Placing the term $\log(3y-1)$ (cf. \eqref{f36}) in the equation above on the left side and $C$ on the right side, taking the exponential, and solving for $y$, we obtain
\begin{multline}
\label{newton}
y=\frac{1}{3}\exp\bigg(-C-x+\frac{1}{5}\log x+\left(\sqrt{3}-\frac{2\sqrt{3}}{5x}\right)\arctan \left(\frac{6y+1}{\sqrt{3}}\right)\\
+\frac{1}{2}\log(9y^2+3y+1)+\frac{1}{x}\left(\frac{27y^2}{5(1-27y^3)}+\frac{1}{25}+O(1/x)\right)\bigg)+\frac{1}{3}
.\end{multline}

The reason for taking the exponential in (\ref{newton}) is to take
care of the branching due to $\log x$, whereas the other $\log$ and $\arctan$ do not matter
since the solution does not encircle their singularities. Equation \eqref{newton} contains, in an implicit form, the solution $y$ to two orders in $x$. $y$ can
be determined from this implicit equation in a number of ways; we chose,
for simplicity to numerically solve the implicit equation using Newton's method. The solution is plotted
in Fig. \ref{fig:abel4}, where we take
$C=2.18-4.65i$ and calculate the solution for the second half of the
path corresponding to $|x|>61.4$. Note that the relative error is
within $1.5\%$.

Since the accuracy of the formal constant of motion is unaffected by
going along the solution path as long as $|x|$ is large, we can obtain
quantitative behavior of the solution for very large $|x|$. By
contrast, in a numerical approach,
the further one integrates along
the path, the less accurate the calculated solution becomes.

\begin{figure}
\includegraphics[scale=0.64]{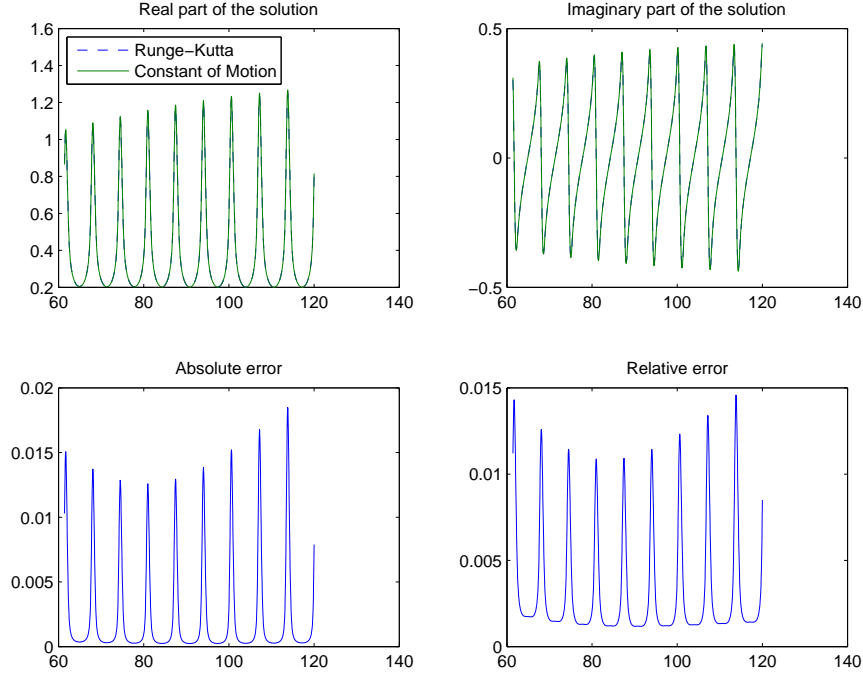}
\caption{Comparison of solutions obtained numerically by the Runge-Kutta method and using the formal constant of motion .}
\label{fig:abel4}
\end{figure}

\begin{figure}
\includegraphics[scale=0.5]{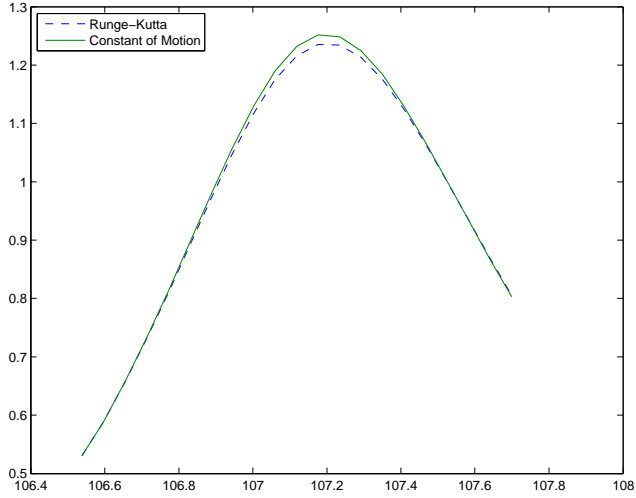}
\caption{A small section of the left-top plot in Fig. \ref{fig:abel4}.}
\label{fig:abel4z}
\end{figure}

\subsection{Finding the positions of the singularities}
We illustrate how to find singularities of the Abel's equation using Proposition \ref{sincom}.
It is known \cite{Invent} that there are only square root singularities, and they appear in two arrays.

For simplicity we choose a simple singular path along which $y$ goes to $+\infty$.

According to Proposition \ref{sincom} we have
\begin{multline}\label{F0}
F_0(y)=-\int_{\infty}^{y}\frac{1}{-3s^3+\frac{1}{9}}ds\\=-\sqrt{3}\arctan (\frac{6y+1}{\sqrt{3}})+\log(3y-1)-\frac{1}{2}\log(9y^2+3y+1)+\frac{\sqrt{3}\pi}{2}\\
F_1(y)=-\frac{1}{5}\int\frac{y}{(-3y^3+\frac{1}{9})^2}dy\\
=\frac{-\frac{54y^2}{1-27y^3}+2\sqrt{3}\arctan (\frac{6y+1}{\sqrt{3}})+2\log(3y-1)-\log(9y^2+3y+1)}{10}
.\end{multline}

Thus the position of the singularity is given by the formula
\begin{multline}
\label{asing}
x_1 =C+o(1)=x_0-\sqrt{3}\left(1-\frac{1}{5x_0}\right)\left(\arctan \left(\frac{6y_0+1}{\sqrt{3}}\right)-\frac{\pi}{2}\right)\\
+\left(1+\frac{1}{5x_0}\right)\left(\log(3y_0-1)-\frac{1}{2}\log(9y_0^2+3y_0+1)\right)-\frac{27y_0^2}{5x_0(1-27y_0^3)}+o(1)
,\end{multline}
 where the initial condition $(x_0,y_0)$ satisfies $|x_0|$ is large and $y_0$ is not close to any of the three roots. We note that the presence of the arctan in the leading order implies that the solutions
remain quasi-periodic beyond the domain accessible to the methods in
 \cite{Invent}.
In (\ref{asing}) we have the freedom of choosing branch of $\log$ and $\arctan$, which enables us to find arrays of singularities.

For example, the position of a singularity corresponding to the
initial condition $x_0=10+60i,~y_0=0.7+0.3i$, calculated using
(\ref{asing}) is $x_1=9.80628+60.2167i$, which is accurate with six significant digits, as checked numerically.

The detailed behavior of the solution near the singularity can be
found by expanding the right hand side of (\ref{asing}). We omit the
calculation here since there are many other methods to determine this
behavior (cf. \cite{Invent}) and it is of lesser importance to the
paper.

\subsection{Connecting regions of transseries}\label{phase}

We choose a path consisting of line segments The path in $x$ consists of line segments connecting $50i$, $50$, $-50i$, $-50$, $50i$, $50$, $-50i$, and $-50(\sqrt{3}+i)$. This corresponds to an angle of $2\pi$ in the original variable, with initial condition $y(50i)=0.6$.

Along this path, the solution of (\ref{abel}) approaches all three complex cube roots of $1/27$. For instance, the root $1/3$ is approached when $x$ traverses the first quadrant along the first segment, the root ${(-1)^{4/3}}/{3}$ is approached when $x$ goes to the lower half plane, and the root ${(-1)^{2/3}}/{3}$ is approached when $x$ goes back to the upper half plane. Some of these values are approached more than once along the entire path. This behavior can easily be shown using the phase portrait of $G_0$, cf. \eqref{eq:eqy} and Corollary \ref{C1}.

Note that along a straight line $x=x_0+x e^{t i}$  where the angle $t$ is fixed the leading term (with only $G_0$ on the right hand side) of the ODE (\ref{abel}) can be written as
$$\frac{d y}{d x}=e^{t i}\left(-3y^3-\frac{y}{5(x_0+x e^{t i})}+\frac{1}{9}\right).$$

Denoting $y_1=\Re{y}$ and $y_2=\Im{y}$, we have

$$\left\{
  \begin{array}{ll}
    \dfrac{d y_1}{d x}=-3 y_1^3\cos t-3 y_2^3\sin t+9  y_1 y_2^2\cos t+9  y_1^2 y_2\sin t+\ds\frac{\cos t}{9}\\
    \dfrac{d y_2}{d x}=-3 y_1^3\sin t+3 y_2^3\cos t+9  y_1 y_2^2\sin t-9  y_1^2 y_2\cos t+\ds\frac{\sin t}{9}
  \end{array}
\right.$$

We can then analyze the phase portraits. For the purpose of illustration, we show some of them in Fig \ref{fig:abel7} and \ref{fig:abel8}.

\begin{figure}[ht!]
\includegraphics[scale=0.8]{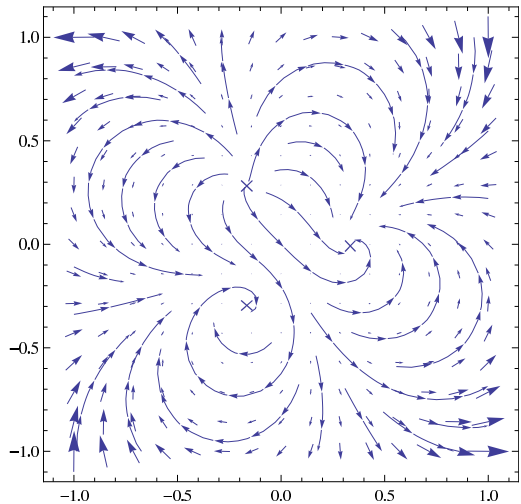}
\caption{Phase portrait of $\Re(y)$ and $\Im(y)$ for $t=-\pi/4$. The $``\times"$ marks are the three roots.}
\label{fig:abel7}
\end{figure}

On  the line segment connecting $50i$ and $50$, it is clear that the initial condition $0.6$ is in the basin of attraction of $1/3$ (cf. Fig. \ref{fig:abel7}).

\begin{figure}[ht!]
\includegraphics[scale=0.8]{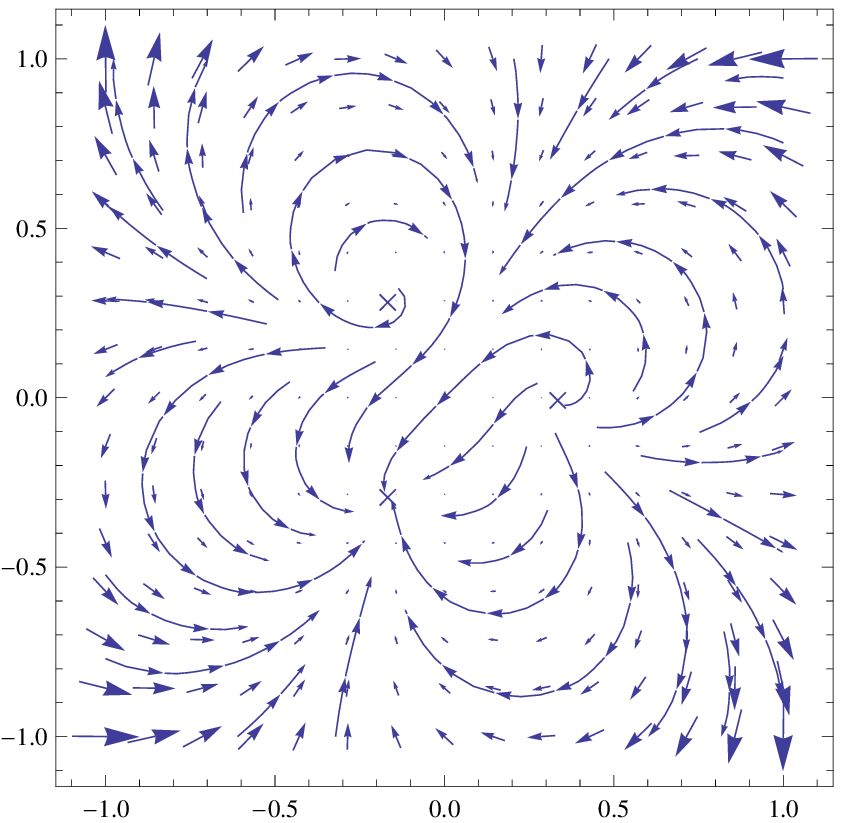}
\caption{Phase portrait of $\Re(y)$ and $\Im(y)$ for $t=5\pi/4$.}
 \label{fig:abel8}
\end{figure}
Since the only stable equilibrium is $a_0=\ds\frac{(-1)^{4/3}}{3}$, on the line segment connecting $50$ and $-50i$  the solution converges to $a_0$ (cf. Fig. \ref{fig:abel8}).

Numerical calculations confirm this, (cf. Fig \ref{fig:abel6}).

\begin{figure}[ht!]
\includegraphics[scale=0.6]{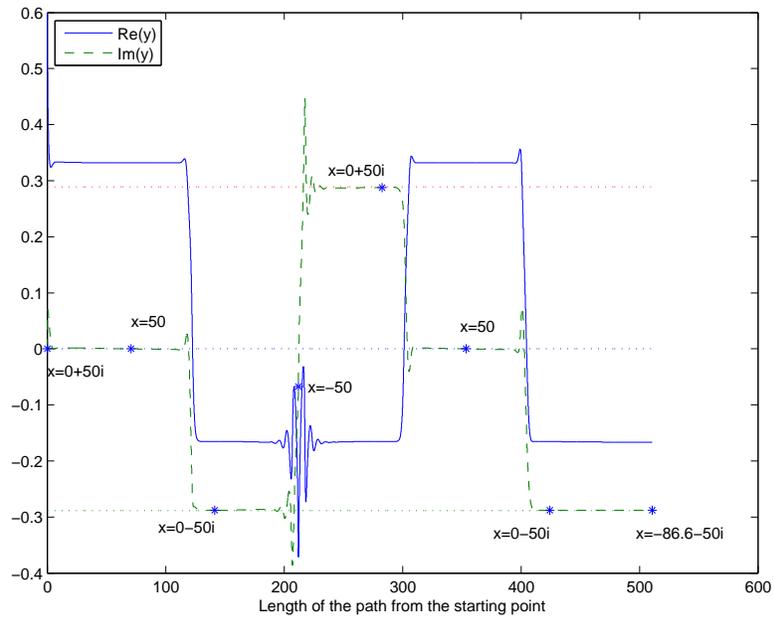}
\caption{Behavior of the solution across transseries regions. Dotted horizontal lines are the imaginary parts of the three roots. The horizontal line is arclength. The path in $x$ consists of line segments connecting $50i$, $50$, $-50i$, $-50$, $50i$, $50$, $-50i$, and $-50(\sqrt{3}+i)$. This corresponds to an angle of $2\pi$ in the original variable.}
\label{fig:abel6}
\end{figure}
\begin{Note}[Absence of limit cycles]{\rm }
  Finally, note that there cannot be a limit cycle in the phase
portraits drawn if $x$ goes along a straight line. If the solution
$y$ approaches a limit cycle, it must lie in an $R$-domain. Thus the
formal constant motion formula (\ref{eq:def2c}) is valid, and the
first term $F_0$ specifies a direction for $x$. If $x$ goes strictly
along this direction towards $\infty$ then the term $a\log x$, which
does not vanish in our case, will go to $\infty$, contradicting the
results about the constant of motion. On the other hand, if $x$ goes
in a different direction, then $-x+F_0(y)$ goes to $\infty$ much
faster than $a\log x$, again a contradiction.
\end{Note}

\subsection{Extension to higher orders} For higher orders, such as the
Painlev\'e equations P1 and P2, a similar procedure works, though the details
are quite a bit more complicated, and we leave them for a subsequent work.
We illustrate, without proofs, the results  for $P1$, $y''=6y^2+z$. Now,  there are two asympotic constants of motion,
as expected. The normal form we work with is $u''+u'x^{-1}-u-u^2/2-392x^{-4}/625=0$. Denoting by $s$ the ``energy of elliptic functions'' $s={u'}^2/2-u^3/3+u^2$
(it turns out that $s$ is one of the bicharacteristic variables of the
sequence of now PDEs governing the terms of the expansion; thus the pair
$(u,s)$ is preferable to $(u,u')$),  one constant
of motion has the asymptotic form
$$C_1=x-L(s,u)+x^{-1}K_1(s,u)+\cdots$$
In the above, denoting  $R=\sqrt{u^3/3+u^2+s}$, $L$ is an incomplete elliptic integral, $L=\int R^{-1}(s,u)du$
and the integration is following a path
winding around the zeros of $R$. The functions $K_1$, $K_2$, $\cdots$ have similar but
longer expressions. We note the absence of a term of the form $a\log x$ (the reason for this is easy to see once the calculation is performed). A second
constant can now be obtained by reduction of order and applying
the first order techniques, or better, by the ``action-angle'' approach
described in the introduction. It is of the form
$$C_2=xJ(s)+[L(s)J(u,s)-J(s)L(u,s)]+x^{-1}\tilde{K}_1+\cdots$$
where $J(u,s)=\int R(s,u)du$; when the variable $u$ is missing from
$J(u,s)$ or $R(u,s)$, this simply means that we are dealing with
complete elliptic integrals. There is directionality in the
asymptotics, as the loops encircling the singularities need to be
rigidly chosen according to the asymptotic direction studied. A
slightly different representation allows us to calculate the constants
to all orders.  Because of directionality, a different asymptotic
formula exists and is more useful for the ``lateral connection'', that
is, for calculating the solution along a circle of fixed but large
radius, which will be detailed in a separate paper, as part of the
Painlev\'e project, see e.g. \cite{Painleve22}.

\section{Acknowledgments} The authors are very grateful to R. Costin for
a careful reading of the manuscript and numerous useful suggestions.  OC's work was partially supported by the NSF grants DMS 0807266 and  DMS  0600369.

\end{document}